\documentclass[centertags,12pt]{amsart}

\usepackage{amssymb}
\usepackage{latexsym}
\usepackage{amsfonts}
\usepackage{amssymb}
\usepackage{graphicx} 
\usepackage{pdflscape} 
\usepackage{diagbox} 

\numberwithin{equation}{section}

\makeatletter
\renewcommand{\subsection}{\@startsection
{subsection}{2}{0mm}{\baselineskip}{-0.25cm}
{\normalfont\normalsize\em}}
\makeatother

\newtheorem{theorem}{Theorem}[section]
\newtheorem{proposition}[theorem]{Proposition}

{\theoremstyle{definition}

\newtheorem{example}[theorem]{Example}
\newtheorem{Ex2.1}[theorem]{Example 2.5 revisited}
\newtheorem{Ex2.2}[theorem]{Example 2.6 revisited}

}

{\theoremstyle{remark}
\newtheorem{remark}[theorem]{Remark}

}

\def\N{\mathbb N}

\def\S{\mathcal S}

\def\1{\mathbf 1}

\def\x{\mathbf x}

\def\cS{\mathcal S}

\def\N {\mathbb{N}}

\def\g {\gamma}

\title[Counting numerical semigroups by genus and even gaps via Kunz-coordinate vectors]{Counting numerical semigroups by genus and even gaps via Kunz-coordinate vectors}

\author[M. Bernardini]{Matheus Bernardini}
\address{}
\email{matheusbernardini@unb.br}

\thanks{{\em 2010 Math. Subj. Class.}: Primary 20M14; 
Secondary 05A15, 05A19}

\thanks{{\em Keywords}: numerical semigroup, multiplicity, even gap, genus, Ap\'ery set, Kunz-coordinate vector}

\begin{document}


\begin{abstract}
We contruct a one-to-one correspondence between a subset of numerical semigroups with genus $g$ and $\gamma$ even gaps and the integer points of a rational polytope. In particular, we give an overview to apply this correspondence to try to decide if the sequence $(n_g)$ is increasing, where $n_g$ denotes the number of numerical semigroups with genus $g$.
\end{abstract}

\maketitle

\section{Introduction}\label{s1}

A {\em numerical semigroup} $S$ is a subset of $\mathbb{N}_0$ such that $0 \in S$, it is closed under addition and the set $G(S):=\mathbb{N}_0\setminus S$, the set of {\em gaps} of $S$, is finite.  The number of elements $g=g(S)$ of $G(S)$ is called the {\em genus} of $S$ and the first non-zero element in $S$ is called the {\em multiplicity} of $S$. If $S$ is a numerical semigroup with genus $g$ then one can ensure that all gaps of $S$ belongs to $[1,2g]$; in particular, $\{2g+i:i\in\N_0\} \subseteq S$ and the number of numerical semigroups with genus $g$, denoted by $n_g$, is finite. Some excellent references for the background on numerical semigroups are the books \cite{GS-R} and \cite{RA}.

Troughtout this paper, we keep the notation proposed by Bernardini and Torres \cite{MF}: the set of numerical semigroups with genus $g$ is denoted by $\cS_g$ and has $n_g$ elements and the the set of numerical semigroups with genus $g$ and $\gamma$ even gaps is denoted by $\cS_\gamma(g)$ and has $N_\gamma(g)$ elements.

In this paper we use the quite useful parametrization
\begin{equation}
\x_g: \cS_\gamma(g)\to \cS_\gamma, S\mapsto S/2,
\label{param}
\end{equation}
where $S/2 := \{s \in \mathbb{N}_0: 2s \in S\}$.

Naturally, the set $\S_\gamma(g)$ and the map $\x_g$ can be generalized. Let $d > 1$ be an integer. The set of numerical semigroups with genus $g$ and $\gamma$ gaps which are congruent to 0 modulo $d$ is denoted by $\cS_{(d,\g)}(g)$. There is a natural parametrization given by
$$\x_{g_d}: \cS_{(d,\gamma)}(g)\to \cS_\gamma, S\mapsto S/d,$$
where $S/d := \{s \in \mathbb{N}_0: ds \in S\}$. This concept appears in \cite{R-GS-GG-B}, for instance.

In this paper, we obtain a one-to-one correspondence between the set $\x_g^{-1}(T)$ and the integer points of a rational polytope.

As an application of this correspondence, we give a new approach to compute the numbers $N_\gamma(g)$. Our main goal is finding a new direction to discuss the following question.

\begin{equation}
\text{Is it true that } n_g \leq n_{g+1}, \text{ for all } g?
\label{weak}
\end{equation}

The first few elements of the sequence $(n_g)$ are $1,1,2,4,7,12,23,39,67$. Kaplan \cite{Kaplan2} wrote a nice survey on this problem and one can find information of these numbers in Sloane's On-line Encyclopedia of Integer Sequences \cite{Sloane}.

Bras-Amor\'os \cite{Amoros1} conjectured remarkable properties on the
behaviour of the sequence $(n_g)$:

\begin{enumerate}
\item $\lim_{g\to\infty}\frac{n_{g+1}+n_g}{n_{g+2}}=1$;
\item $\lim_{g\to\infty}\frac{n_{g+1}}{n_g}=\varphi:= 
\frac{1+\sqrt{5}}{2}$;
\item $n_{g+2}\geq n_{g+1}+n_g$ for any $g$.
\end{enumerate}

Zhai \cite{Zhai} proved that $\lim_{g\to\infty}n_g\varphi^{-g}$ is a constant. As a consequence, it confirms that items (1) and (2) hold true. However, item (3) is still an open problem; even a weaker version, proposed at (\ref{weak}), is an open question. Zhai's result also ensures that $n_g < n_{g+1}$ for large enough $g$. Fromentin and Hivert \cite{FH} verified that $n_g < n_{g+1}$ also holds true for $g \leq 67$. 

Torres \cite{Torres} proved that $\cS_{\g}(g) \neq \emptyset$ if, and only if, $2g \geq 3\g$. Hence,

\begin{equation}
n_g=\sum_{\gamma=0}^{\lfloor 2g/3\rfloor}N_\gamma(g)\, .
\label{ng}
\end{equation}

In order to work on Question (\ref{weak}), Bernardini and Torres \cite{MF} tried to understand the effect of the even gaps on a numerical semigroup. By using the so-called $t$-translation, they proved that $N_\g(g) = N_\g(3\g)$ for $g \geq 3\g$ and also $N_\g(g) < N_\g(3\g)$ for $g < 3\g$. Although numerical evidence points out that $N_\g(g) \leq N_\g(g+1)$ holds true for all $g$ and $\g$, their methods could not compare numbers $N_{\gamma}(g_1)$ and $N_{\gamma}(g_2)$, with $3\gamma/2 \leq g_1 < g_2 < 3\gamma$. Notice that if $N_{\gamma}(g_1) \leq N_{\gamma}(g_2)$, for $3\gamma/2 \leq g_1 < g_2 < 3\gamma$ then $n_g < n_{g+1}$, for all $g$.

\section{Ap\'ery set and Kunz-coordinate vector}

Let $S$ be a numerical semigroup and $n \in S$. The Ap\'ery set of $S$ (with respect to $n$) is the set $Ap(S,n) = \{s \in S: s-n \notin S\}$. If $n = 1$, then $S = \N_0$ and $Ap(\N_0, 1) = \{0\}$. If $n > 1$, then a there are $w_1, \ldots, w_{n-1} \in \N$ such that $Ap(S,n) = \{0, w_1, \ldots, w_{n-1}\}$, where $w_i = \min\{s \in S: s \equiv i \pmod n\}$.

\begin{proposition}
Let $S$ be a numerical semigroup with multiplicity $m$ and $Ap(S,m) = \{0, w_1, \ldots, w_{m-1}\}$. Then
$$S = \langle m, w_1, w_2, \ldots, w_{m-1}\rangle.$$
\end{proposition}

\begin{proof}
It is clear that $am \in \langle m, w_1, w_2, \ldots, w_{m-1}\rangle, \forall a \in \mathbb{N}.$ For $s \in S$, $m \nmid s$, there is $\tilde{k} \in \N_0$ such that $s = w_i + \tilde{k}m \in \langle m, w_1, w_2, \ldots, w_{m-1}\rangle$. On the other hand, $m, w_1, \ldots, w_{m-1} \in S$.
\end{proof}

Let $S$ be a numerical semigroup, $n \in S$ and consider $Ap(S,n) = \{0, w_1, \ldots, w_{n-1}\}$. There are $e_1, \ldots, e_{n-1} \in \N$ such that $w_i = ne_i + i$, for each $i \in \{1, \ldots, n-1\}$. The vector $(e_1, \ldots, e_{n-1}) \in \N_0^{n-1}$ is called the \textit{Kunz-coordinate vector} of $S$ (with respect to $n$). In particular, if $m$ is the multiplicity of $S$, then the Kunz-coordinate vector of $S$ (with respect to $m$) is in $\N^{m-1}$. This concept appears in \cite{BP}, for instance.

A natural task is finding conditions for a vector $(x_1, \ldots, x_{m-1}) \in \N^{m-1}$ to be a Kunz-coordinate vector (with respect to the multiplicity $m$ of $S$) of some numerical semigroup $S$ with multiplicity $m$. The following examples illustrate the general method, which is presented in Proposition \ref{ineqgenus}.

\begin{example}
Numerical semigroups with multiplicity $2$ are $\langle 2, 2e_1+1 \rangle$, where $e_1 \in \N$.

There is a one-to-one correspondence between the set of numerical semigroups with multiplicity $2$ and the set of positive integers given by $\langle 2, 2e_1+1 \rangle \mapsto e_1$.
\end{example}

\begin{example}
Let $S = \langle 3, 3e_1+1, 3e_2+2 \rangle$ be a numerical semigroup with multiplicity $3$ and genus $g$, where $e_1, e_2 \in \N$. By minimality of $w_1 = 3e_1 + 1$ and $w_2 = 3e_2 + 2$, $(e_1,e_2)$ satisfies
$$
\begin{cases}
(3e_1 + 1) + (3e_1 + 1) \geq 3e_2 + 2 \\
(3e_2 + 2) + (3e_2 + 2) \geq 3e_1 + 1.
\end{cases}
$$
The set of gaps of $S$ has $e_1+e_2$ elements, since $G(S) = \{3n_1+1: 0 \leq n_1 < e_1 \} \cup \{3n_2+2: 0 \leq n_2 < e_2\}$. Thus, $e_1 + e_2 = g$. On the other hand, if $(e_1,e_2) \in \N^2$ is such that $2e_1 \geq e_2$, $2e_2 + 1 \geq e_1$ and $e_1+e_2 = g$, then $\langle 3, 3e_1+1, 3e_2+2 \rangle$ is a numerical semigroup with multiplicity $m$ and genus $g$.

Hence, there is a one-to-one correspondence between the set of numerical semigroups with multiplicity $3$ and the vectors of $\N^2$ which are solutions of

$$
\begin{cases}
2X_1 \geq X_2 \\
2X_2 +1 \geq X_1 \\
X_1 + X_2 = g.
\end{cases}
$$
\label{m=3}
\end{example}

In order to give a characterization of numerical semigroups with fixed multiplicity and fixed genus, the main idea is generalizing Example \ref{m=3}. The following is a result due to Rosales et al. \cite{R-GS-GG-B}.

\begin{proposition}
There is a one-to-one correspondence between the set of numerical semigroups with multiplicity $m$ and genus $g$ and the positive integer solutions of the system of inequalities

$$
\begin{cases}
X_i + X_j \geq  X_{i+j}, \hspace{1.2cm} \text{ for } 1 \leq i \leq j \leq m-1; i + j < m; \\
X_i + X_j + 1 \geq X_{i+j-m}, \text{ for } 1 \leq i \leq j \leq m-1; i + j > m \\
\sum_{k=1}^{m-1} X_k =  g.
\end{cases}
$$
\label{ineqgenus}
\end{proposition}

Let $S = \langle m, w_1, \ldots, w_{m-1} \rangle$ be a numerical semigroup with multiplicity $m$ and genus $g$, where $w_i = me_i + i$. The main idea of the proof is using the minimality of $w_1, \ldots, w_{m-1}$ and observing that $w_i + w_j \equiv i+j \pmod m$ and $G(S) = \bigcup_{i=1}^{m-1} \{mn_i + i: 0 \leq n_i < e_i \}$. For a full proof, see \cite{R-GS-GG-B}.

\section{The main result and an application to a counting problem}\label{s3}

In \cite{MF}, the calculation of $N_\gamma(g)$ was given by

\begin{equation}
N_\gamma(g) = \sum_{T \in \mathcal{S}_\gamma}\#\mathbf{x}_g^{-1}(T).
\label{MF_even}
\end{equation}

In this section, we present a new way for computing those numbers. In order to do this, we fix the multiplicity of $T \in \mathcal{S}_\gamma$.

First of all, we obtain a relation between the genus and the multiplicity of a numerical semigroup. 

\begin{proposition}
Let $S$ be a numerical semigroup with genus $g$ and multiplicity $m$. Then $m \leq g+1$.
\label{multiplicity}
\end{proposition}

\noindent \textbf{\textit{Proof}}
If a numerical semigroup $S$ has multiplicity $m$ and genus $g$ with $m \geq g+2$, then the number of gaps of $S$ would be, at least, $g+1$ and it is a contradiction. Hence $m \leq g+1$.

\begin{remark}
The bound obtained in Proposition \ref{multiplicity} is sharp, since $\{0, g+1, \ldots \}$ has genus $g$ has multiplicity $g+1$.
\end{remark}

If $\gamma = 0$, then $\mathcal{S}_0 = \{\mathbb{N}_0\}$ and $\mathbf{x}_g^{-1}(\mathbb{N}_0) = \{\langle 2, 2g+1 \rangle\}$. Hence, $N_0(g) = 1$, for all $g$. If $\gamma > 0$, we divide the set $\mathcal{S}_\gamma $ into the subsets $\mathcal{S}_{\gamma}^m := \{S: g(S) = \gamma \text{ and } m(S) = m\}$, where $m \in [2,\gamma+1] \cap \mathbb{Z}$. We can write

\begin{equation}
\mathcal{S}_\gamma = \bigcup_{m=2}^{\gamma+1} \mathcal{S}_{\gamma}^m.
\label{nova}
\end{equation}

Putting (\ref{MF_even}) and (\ref{nova}) together, we obtain

$$N_{\gamma}(g) = \sum_{m = 2}^{\gamma+1}\sum_{T \in \mathcal{S}_{\gamma}^m} \#\mathbf{x}_g^{-1}(T).$$

Thus, it is important to give a characterization for $T \in \mathcal{S}_\gamma^m$. We can describe $T$ by its Ap\'ery set (with respect to its multiplicity $m$) and write 
$$T = \langle m, me_1 + 1, me_2 + 2, \ldots, me_{m-1} + (m-1) \rangle,$$ 
where $me_i +i = \min\{s \in S: s \equiv i \pmod m\}$.

The next result characterizes all numerical semigroups of $\mathbf{x}_g^{-1}(T)$, for $T \in \mathcal{S}_\gamma^m$. It is a consequence of Proposition \ref{ineqgenus}.

\begin{theorem}
Let $T = \langle m, me_1+1, \ldots, me_{m-1}+ (m-1)\rangle \in \mathcal{S}_{\gamma}^m$. A numerical semigroup $S$ belongs to $\mathbf{x}_g^{-1}(T)$ if, and only if, 
$$S = \langle 2m, 2me_1+2, \ldots, 2me_{m-1}+(2m-2), 2mk_1+1, 2mk_3+3, \ldots, 2mk_{2m-1}+(2m-1) \rangle,$$ 
where $(k_1,k_3,\ldots,k_{2m-1}) \in \mathbb{N}_0^{m}$ satisfies the system

\begin{equation*}
\begin{cases}
(*)
\begin{cases}
X_{2i-1} + e_j \geq X_{2(i+j)-1}, \hspace{1.6cm} \text{ for } 1 \leq i \leq m; 1 \leq j \leq m-1; i + j \leq m; \\
X_{2i-1} + e_j + 1 \geq X_{2(i+j-m)-1}, \hspace{0.4cm} \text{ for } 1 \leq i \leq m; 1 \leq j \leq m-1; i + j > m; \\
\end{cases} \\
(**)
\begin{cases}
X_{2i-1} + X_{2j-1} \geq e_{i+j-1}, \hspace{1.2cm} \text{ for } 1 \leq i \leq j \leq m; i + j \leq m; \\
X_{2i-1} + X_{2j-1} + 1 \geq e_{i+j-1-m}, \text{ for } 1 \leq i \leq j \leq m; i + j \geq m+2; \\
\end{cases} \\
\sum_{i=1}^{m} X_{2i-1} = g-\gamma,
\end{cases}
\label{system}
\end{equation*}
\label{oddeven}
\end{theorem}

\noindent \textbf{\textit{Proof}}
The even numbers $2m, 2me_1 + 2, \ldots, 2me_{m-1} + 2me_{m-1} + 2(m-1)$ belongs to  $Ap(2m,S)$. Let $2mk_1+1, 2mk_3+3, \ldots, 2mk_{2m-1}+(2m-1)$ be the odd numbers of $Ap(2m,S)$. Thus, $(e_1, k_1, e_2, k_3, \ldots, e_{m-1}, k_{2m-1}) \in \mathbb{N}_0^{m-1}$ is the Kunz-coordinate vector of $S$ (with respect to $2m$). 

Now, we apply Proposition \ref{ineqgenus}. Inequalites given in $(*)$ come from sums of an odd element of $Ap(2m,S)$ with an even element of $Ap(2m,S)$, while inequalities given in $(**)$ come from sums of two odd elements of $Ap(2m,S)$. Since $(e_1, \ldots, e_{m-1})$ is the Kunz-coordinate vector of $T$ (with respect to $m$), then the sum of two even elements of $Ap(2m,S)$ belongs to $S$. Finally, last equality comes from the fact that $S$ has $g - \gamma$ odd gaps. 

\begin{remark}
Some of the numbers $k_i$ can be zero. Hence, it is possible that the multiplicity of $S$ is not $2m$.
\end{remark}

\noindent \textbf{Example 3} Let $T = \langle 2, 2\gamma+1 \rangle \in \mathcal{S}_\gamma^2$, with $\gamma \in \mathbb{N}$. Theorem \ref{oddeven} ensures that if $S \in \mathbf{x}_g^{-1}(T)$, then

$$S = \langle 4, 4\gamma+2, 4k_1+1, 4k_3 + 3 \rangle,$$
where $(k_1,k_3) \in \mathbb{N}_0^2$ satisties

$$
(\#)\begin{cases}
(*) \begin{cases}
-\gamma-1 \leq X_3 - X_1 \leq \gamma
\end{cases} \\
(**) \begin{cases}
X_1 + X_1 \geq \gamma \\
X_3 + X_3 + 1 \geq \gamma
\end{cases} \\
X_1 + X_3 = g - \gamma.
\end{cases}
$$

\begin{figure}[h]
\includegraphics[scale=.65]{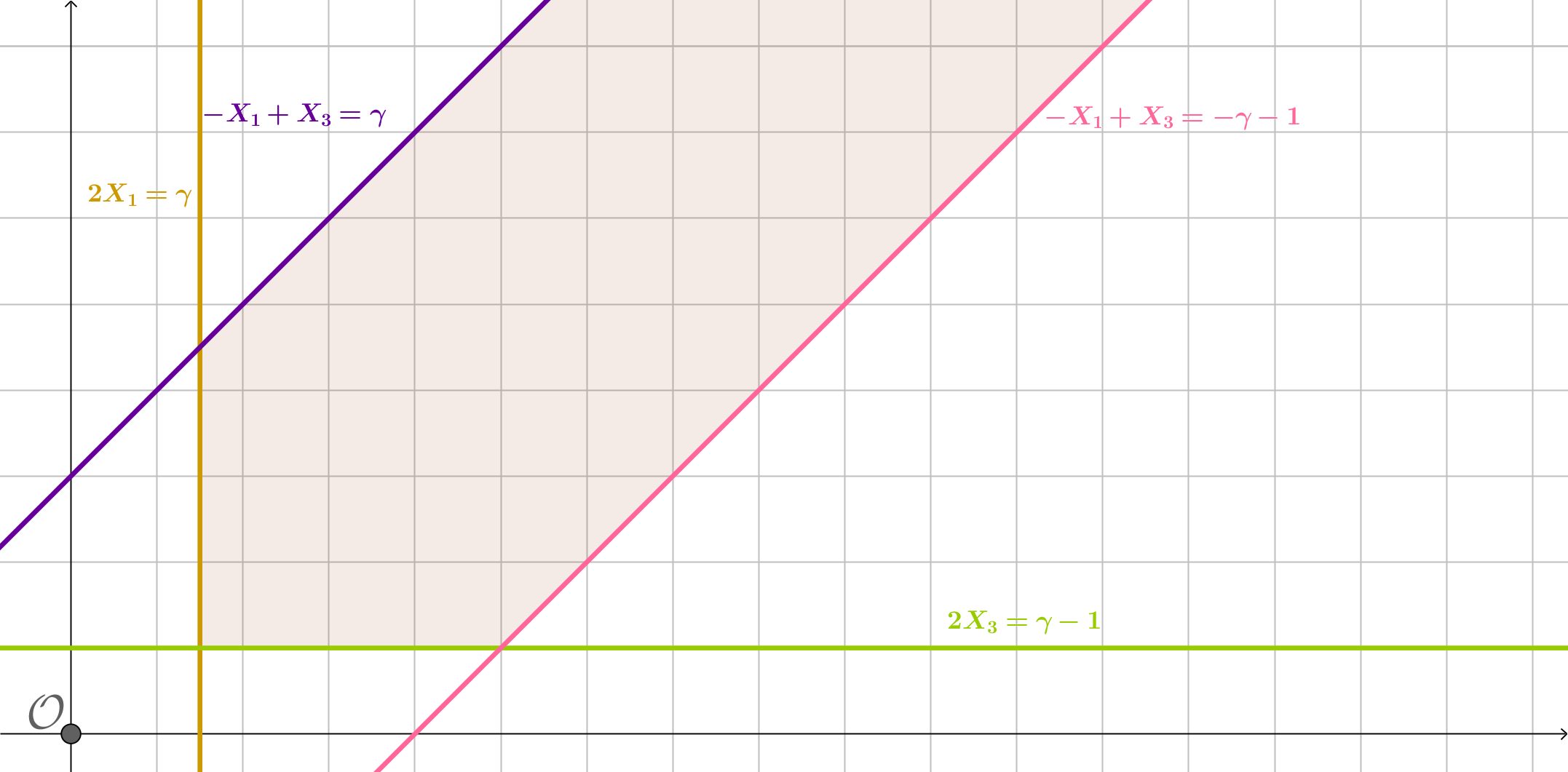}
\caption{Region in $\mathbb{R}^2$ given by inequalities $(*)$ and $(**)$.}
\label{fig:1}
\end{figure}


The set of integer points of this region is in one-to-one correspondence with the set $\{S \in \mathcal{S}_\gamma(g): S/2 \text{ has multiplicity } 2\}$. 

If $g$ is fixed, then the set of points that satisfies the system $(\#)$ is a polytope (a line segment). We are interested in the set of integer points of this polytope. The following figure shows examples for some values of $g$. Each integer point represents a numerical semigroup of the set  $\{S \in \mathcal{S}_\gamma(g): S/2 \text{ has multiplicity } 2\}$.

\begin{figure}[h]
\includegraphics[scale=.65]{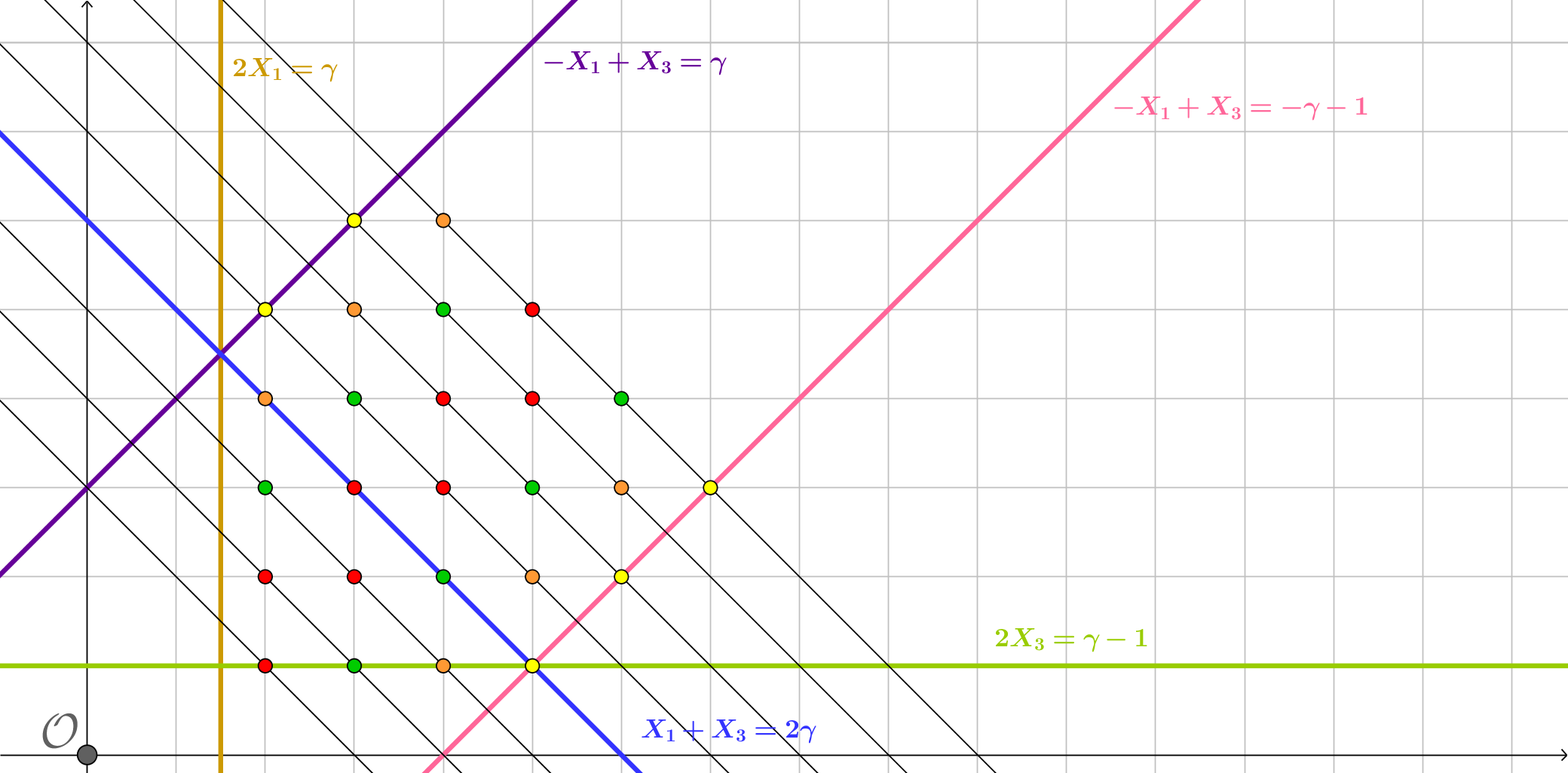}
\caption{For fixed $g$, the integer points in the line segment represent numerical semigroups of $\{S \in \mathcal{S}_\gamma(g): S/2 \text{ has multiplicity } 2\}$.}
\label{fig:2}
\end{figure}

\pagebreak

Let $N^m_\gamma(g) = \sum_{T \in \mathcal{S}_\gamma^m} \#\mathbf{x}_g^{-1}(T)$. After some computations, we obtain

$$N^2_\gamma(g)=
\begin{cases}
0, \text{ if } g < 2\gamma \\
k+1, \text{ if } g = 2\gamma + k \text{ and } k \in \{0, 1, \ldots, \gamma-1\} \\
\gamma+1, \text{ if } g \geq 3\gamma.
\end{cases}$$

In particular, $N^2_\gamma(g) \leq N^2_\gamma(g+1)$. We leave the following open question.

\begin{equation}
\text{Let } \gamma \in \mathbb{N} \text{ and } m \in [2,\gamma+1] \cap \mathbb{Z}. \text{ Is it true that } N^m_\gamma(g) \leq N^m_\gamma(g+1), \text{ for all } g?
\label{new}
\end{equation}

A positive answer to Question (\ref{new}) implies a positive answer to Question (\ref{weak}).

{\bf Acknowledgment.} 
The author was partially supported FAPDF-Brazil (grant 23072.91.49580.29052018). Part of this paper was presented in the  \lq\lq INdAM: International meeting on numerical semigroups" (2018) at Cortona, Italy. I am grateful to the referee for their comments, suggestions and corrections that allowed to improve this version of the paper.


\begin{thebibliography}{99}

\bibitem{MF} M. Bernardini and F. Torres, {\em Counting numerical semigroups by genus and even gaps}, Disc. Math., {\em 340} (2017), 2853--2863.

\bibitem{BP} V. Blanco and J. Puerto, {\em An application of integer programming to the decomposition of numerical semigroups}, SIAM J. Discrete Math., {\em 26(3)} (2012), 1210-–1237. 

\bibitem{Amoros1} M. Bras-Amor\'os, {\em Fibonacci-like behavior of 
the number of numerical semigroups of a given genus}, Semigroup 
Forum {\bf 76} (2008), 379--384.

\bibitem{FH} J. Fromentin and F. Hivert, {\em Exploring the tree of numerical semigroups}, 
Math. Comp. {\bf 85}(301) (2016), 2553--2568.

\bibitem{GS-R} P.A. Garc\'{\i }a-S\'anchez and J.C. Rosales, \lq\lq 
Numerical semigroups", Developments in Mathematics vol. {\bf 20}, 
Springer, New York, 2009.

\bibitem{Kaplan2} N. Kaplan, {\em Counting numerical semigroups}, Amer. Math. Monthly {\bf 163} (2017), 375--384.

\bibitem{RA} J.L. Ram\'{\i }rez-Alfons\'{\i }n, \lq\lq The Diophantine Frobenius Problem", Oxford Univ. Press vol {\bf 30}, 2005.

\bibitem{R-GS-GG-B} J.C. Rosales, P.A. Garc\'{i}a-S\'{a}nchez, J.I. Garc\'{i}a-Garc\'{i}a and M.B. Branco, {\em Systems of inequalities and numerical semigroups}, J. London Math. Soc. (2), {\em 65} (2002), 611--623.

\bibitem{Rosales} J.C. Rosales, P.A. Garc\'{\i }a-S\'anchez,  J.I. Garc\'{\i }a-S\'anchez, J.M. Urbano-Blanco, {\em Proportionally modular Diophantine inequalities}, J. Number Theory {\bf 103} (2003), 281--294.

\bibitem{Sloane} N.J.A. Sloane, \lq\lq The On-Line Encyclopedia of Integer Sequences", A007323, \\http://www.research.att.com/$\sim$njas/sequences/(2009)

\bibitem{Torres} F. Torres, {\em On $\gamma$-hyperelliptic 
numerical semigroups}, Semigroup Forum {\bf 55} (1997), 364--379.

\bibitem{Zhai} A. Zhai, {\em Fibonacci-like growth of numerical 
semigroups of a given genus}, Semigroup Forum {\bf 86} (2013), 
634--662.

\end{thebibliography}
\end{document}